\documentclass[11pt, nonatbib]{elsarticle}
\usepackage{fancyhdr}
\usepackage{marginnote}
\usepackage{todonotes}

\usepackage{amsmath}
\usepackage{mathtools}
\usepackage{amsfonts}
\usepackage{amsthm}
\usepackage{amssymb}
\usepackage{color}
\usepackage{dsfont}
\usepackage{geometry}
\usepackage{comment}
\usepackage{hyperref}
\usepackage{scalerel}[2014/03/10]
\usepackage[usestackEOL]{stackengine}
\newgeometry{tmargin=2.8cm, bmargin=3.5cm, lmargin=1.7cm, rmargin=1.7cm}
\usepackage{color}

\makeatletter
\let\c@author\relax
\makeatother
\newcommand{\Xset}{\mathcal{X}}
\DeclareMathOperator{\prox}{prox}
\definecolor{darkgreen}{rgb}{0.00, 0.50, 0.00}

\usepackage[backend=biber, citestyle=numeric-comp, bibstyle=trad-abbrv, url=false, doi=false, isbn=false]{biblatex}
\newtheorem{theo}{\bf Theorem}[section]
 
\newtheorem{coro}{\bf Corollary}[section]
\newtheorem{lem}[coro]{\bf Lemma} 
\newtheorem{defin}[coro]{\bf Definition}

\newtheorem{prop}[coro]{\bf Proposition}

\numberwithin{equation}{section}

\def\R{{\mathbb{R}}}

\def\N{{\mathbb{N}}}

\def\rd{{\mathbb{R}^{d}}}

\DeclareMathOperator{\dist}{dist}

\def\Rd{\mathbb{R}^d}
\def\ve{{\varepsilon}}

\def\avint{\,\ThisStyle{\ensurestackMath{%
  \stackinset{c}{.2\LMpt}{c}{.5\LMpt}{\SavedStyle-}{\SavedStyle\phantom{\int}}}%
  \setbox0=\hbox{$\SavedStyle\int\,$}\kern-\wd0}\int}
\newcommand{\normalcone}[1]{\operatorname{N}_{#1}}
\begin{document}

\begin{frontmatter}

\title{Convergence of projected stochastic approximation algorithm}

\author[1]{Michał Borowski}\ead{m.borowski@mimuw.edu.pl}
\author[1]{Błażej Miasojedow}\ead{}

\address[1]{Institute of Applied Mathematics and Mechanics,
University of Warsaw, ul. Banacha 2, 02-097 Warsaw, Poland
}

\fntext[myfootnote]{MSC2020: 62L20 (34D05, 46B50)}
\fntext[myfootnote]{B.M. is supported by NCN grant 2018/31/B/ST1/00253 }

\begin{abstract} 
  
We study the Robbins-Monro stochastic approximation algorithm with projections on a hyperrectangle and prove its convergence. This work fills a gap in the convergence proof of the classic book by Kushner and Yin. Using the ODE method, we show that the algorithm converges to stationary points of a related projected ODE. Our results provide a better theoretical foundation for stochastic optimization techniques, including stochastic gradient descent and its proximal version. These results extend the algorithm's applicability and relax some assumptions of previous research.

\end{abstract}

\end{frontmatter}
\section{Introduction}\label{sec:intro}
Stochastic approximation (SA) algorithms are widely used for solving optimization problems and finding roots of functions when the data is noisy or incomplete. These iterative methods update estimates based on observed samples and are fundamental in fields such as machine learning and signal processing. A key challenge in classical SA algorithms is ensuring stability, i.e., that the iterates remain within a compact set. Without stability, the algorithm may diverge, making theoretical analysis difficult. A standard solution is to use a projected version, where the iterates are kept within a predefined set using projections. This paper focuses on proving the convergence of such projected SA algorithms while filling gaps in existing proofs.
 
 We focus on a variant of the Robbins-Monro algorithm that incorporates projections, which is used to approximate a root of a continuous function $h : \rd \to \rd$. The algorithm updates iteratively according to the rule
\begin{equation}\label{eq:def-SA}\tag{SA}
    x_{n+1} = \Pi_K\left(x_n + \gamma_n\Phi(x_n, \xi_n))\right)\,,
\end{equation}
where the sequence $(\gamma_n)_n$ of positive numbers denotes step sizes,  $K$ denotes a hyperrectangle, and $\Pi_K : \Rd \to K$ is the projection onto the set $K$. The term $\Phi(x_n, \xi_n)$ stands for some approximation of $h(x_n)$, possibly depending on some information $\xi_n$ obtained through the proceeding of the algorithm. 

We write $\Phi(x_n, \xi_n) = h(x_n) + e_n + r_n$, meaning that $(e_n + r_n)_n$ are $\rd$-valued random variables modelling the noise. We work under standard assumptions on step sizes and errors, that is
\begin{equation}\tag{Con-$\gamma$}\label{eq:con-gamma}
    \sum_{n=1}^{\infty} \gamma_n = \infty\,, \quad \lim_{n \to \infty} \gamma_n = 0\,, \quad 
\end{equation}
\begin{equation}\tag{Con-$e$-$r$}\label{eq:con-er}
    \lim_{n \to \infty} |r_n| = 0 \text{ almost surely} \quad \text{ and } \quad \sum_{n=1}^{\infty} \gamma_ne_n \text{ converges almost surely.}
\end{equation}
The first assumption of~\eqref{eq:con-gamma} keeps the procedure ongoing, while~\eqref{eq:con-er} ensures that the noise diminishes over time. Note that those assumptions do not refer to the nature of the noise, which can be both implicit and explicit. We also stress that the sequence $(e_n)_n$ does not have to be bounded.

Putting in~\eqref{eq:def-SA} function $h=-\nabla f$ for some convex and smooth function $f$, we get the minimization algorithm for function $f$, i.e., stochastic gradient descent, which is a foundational tool in stochastic optimization. Its applications range from machine learning~\cite{Bottou, Bottou2, Zinkevich}, including deep learning~\cite{Cao}, reinforcement learning~\cite{Li}, multi-agent systems~\cite{Bianchi}, signal processing~\cite{Borkar2, Pereyra}, models with incomplete data~\cite{Delyon, Fort-Moulines, Andrieu-Moulines, Blazej}, and more. See in particular applications of projected~\eqref{eq:def-SA}, for instance~\cite{Beck, Duchi, Gupta}. Even in the case of $f$ being non-smooth/non-convex, the method can be adapted, obtaining so-called proximal stochastic gradient descent, see~\cite{Majewski, Parikh, Bolte} and references therein. 

Despite widespread application, the theoretical properties of~\eqref{eq:def-SA}, particularly convergence, are not fully understood. The classical text by Kushner and Yin~\cite{Kushner} is often cited as providing a comprehensive analysis of these algorithms, see for instance~\cite{Bianchi, Larson}. However, a closer examination reveals gaps in the proof of convergence, i.e., proof of~\cite[Theorem 5.2.1]{Kushner}. To the author's best knowledge, the complete proof is missing in the literature.

The primary objective of this paper is to address this gap by providing rigorous proof of convergence for the projected Robbins-Monro algorithm under the assumption that the constraint set is a hyperrectangle. We stress that from the view of the applicability of the method, this assumption is not very restrictive, as the set $K$ is introduced to the technique mainly to guarantee stability (boundedness) of the iteration~\eqref{eq:def-SA}, meaning that the exact structure of $K$ is not crucial. In fact, hyperrectangles have an advantage over other sets and offer computational simplicity in projecting. Nonetheless, we stress that the case of the convergence of~\eqref{eq:def-SA} for an arbitrary convex and bounded set $K$, claimed in~\cite[Theorem 5.2.1]{Kushner}, remains unsolved.

Our analysis builds upon the ODE (Ordinary Differential Equation) method, which views the stochastic approximation as a discretization of a projected ODE, see Section~\ref{sec:meth} for more details. Following this approach, we establish the equicontinuity of the relevant sequences (Theorem~\ref{theo:main}) and demonstrate that the algorithm converges to the stationary points of the projected ODE, see Theorem~\ref{theo:final}.

Our contributions are twofold. First, we provide a detailed proof of convergence for the Robbins-Monro algorithm with projection onto a hyperrectangle, filling a critical gap in the existing literature. Second, we discuss the implications of our results to non-convex and non-smooth optimization, in particular, proximal stochastic gradient descent and stochastic subgradient algorithms. This includes relaxing some assumptions in prior work~\cite{Majewski}, thereby broadening the applicability of the result.

The remainder of the paper is organized as follows. Section~\ref{sec:meth} reviews the methodology of the ODE approach and identifies the specific gaps in the proof of~\cite[Theorem 5.2.1]{Kushner}. We fill this gap by proving Theorem~\ref{theo:main}, whose proof is the main point of Section~\ref{sec:res}. In section~\ref{sec:app}, we present the convergence result for~\eqref{eq:def-SA} in Theorem~\ref{theo:final} and discuss additional applications of Theorem~\ref{theo:main}, including the convergence of the stochastic proximal gradient descent algorithm, detailed in Theorem~\ref{thm:stoch-prox-grad-conv}.
\section{Preliminary concepts}\label{sec:meth}
In this section, we shall discuss the general methodology of analyzing the convergence of the~\eqref{eq:def-SA} algorithm under assumptions~\eqref{eq:con-gamma} and~\eqref{eq:con-er}, especially the approach described in~\cite[Theorem 5.2.1]{Kushner}. The analysis is based on the so-called ODE method, dating back to~\cite{Ljung, Kushner2, Borkar}. The main point of the method is to view~\eqref{eq:def-SA} as a discretization of a continuous process — projected ODE. As step sizes in~\eqref{eq:def-SA} tend to $0$, one expects that the iteration's limiting behavior matches the behavior of the ODE. In particular, it should be possible to identify this limiting behavior as a solution to the ODE. On the other hand, if all trajectories of the ODE converge to some point — stationary state — iteration~\eqref{eq:def-SA} should converge to this point.

Let us now express it more formally. We stress that the set $K$ shall always be a rectangle, i.e., $K = \prod_{i=1}^{d} [a_i, b_i]$, where $a_i < b_i$ for every $i$. At first, having the procedure~\eqref{eq:def-SA}, let us define the sequence of projections $(P_n)_n$ as
\begin{equation*}
    P_n = x_n + \gamma_n(e_n + r_n + h(x_n)) - \Pi_K\left(x_n + \gamma_n(e_n + r_n + h(x_n))\right)\,.
\end{equation*}
Then, we may write
\begin{equation*}
    x_{n+1} = x_n + \gamma_n(e_n + r_n + h(x_n)) - P_n\,,
\end{equation*}
and moreover
\begin{equation}\label{eq:projtheo2}
    x_{n+1} = x_1 + \sum_{k=1}^{n} \left(\gamma_k(e_k + r_k + h(x_k)) - P_k\right)\,.
\end{equation}
Let us define the sequence $(t_n)_{n \in \N}$ as $t_n \coloneqq \sum_{k=1}^{n} \gamma_k$, and the sequences $(X_n)_n$ and $(Z_n)_n$ as
\begin{align}
    X_1(t) &= x_k\, \quad \text{ for } t \in [t_{k-1}, t_k) \quad \text{ and } \quad X_n(t) = X_1(t + t_{n-1})\, \label{eq:def-f} \\
    Z_1(t) &= \sum_{s=1}^{k} P_s \quad \text{ for } t \in [t_{k-1}, t_k) \quad \text{ and } \quad Z_n(t) = Z_1(t + t_{n-1})\,. \label{eq:def-Z}
\end{align}
According to~\eqref{eq:def-f}, $X_1$ is defined as a piecewise constant for each $k \in N$ obtaining value $x_k$ on the interval of length $\gamma_k$. On the other hand, functions $(X_n)_n$ are successive translations of the function $X_1$, approximating its limiting behavior. Let us remark that although we defined $(X_n)_n$, $(Z_n)_n$ as piecewise constant, one can also define them as piecewise linear and, therefore, continuous without essentially affecting the whole approach.

Following the ODE method, we want to obtain from the sequences $(X_n)_n$ and $(Z_n)_n$ the solution of the following projected ODE
\begin{equation}\label{eq:ode}
    \dot x = h(x) - z\,, \quad z \in N_K(x)\,.
\end{equation}
Here $N_K(a)$, for $a \in \Rd$, denotes (outer) normal cone to the set $K$ in point $a$, i.e.,
\begin{equation}\label{eq:defnorm}
    N_K(a) \coloneqq \{b \in \Rd \quad | \quad \forall {c \in K} : \langle b, a - c \rangle \geq 0\}\,.
\end{equation}
We say that pair $(X, z)$, $X, z : 
\R \to \Rd$ is a solution to~\eqref{eq:ode} if $X$ is differentiable, $z$ is measurable, $X'(t) = h(X(t)) - z(t)$ for a.e. $t \in \R$, and $z(t) \in N_K(X(t))$ for a.e. $t \in \R$. Let us stress that one could also define projected ODE as $\dot x = \Pi_{T_K(x)}(h(x))$. Here $T_K(a)$, for $a \in \Rd$, denotes the tangent cone to the set $K$ in point $a$, i.e.,
\begin{equation}\label{eq:deftan}
    T_K(a) \coloneqq \{b \in \Rd \quad | \quad \forall {c \in N_K(a)} : \langle b, c \rangle \leq 0\}\,.
\end{equation}
To get a solution of~\eqref{eq:ode} from sequences $(X_n)_n$ and $(Z_n)_n$, one can use a generalization of Arzel\'a-Ascoli theorem, using the following definition of equicontinuity. 
\begin{defin}[Equicontinuity in extended sense]
    The sequence $(f_n)_n$, $f_n : [0, \infty) \to \Rd$ is equicontinuous in extended sense if for every $\ve > 0$ and $T > 0$, there exists $\delta > 0$ such that
    \begin{equation*}
        \limsup_{n \to \infty} |f_n(t) - f_n(s)| < \ve\,,\quad \text{whenever }\quad |t - s| < \delta \text{ and } s,t \leq T\,.
    \end{equation*}
\end{defin}
Arzel\`a--Ascoli--type theorem using the definition above reads as follows.
\begin{theo}[Corollary IV.8 in \cite{Dunford}]\label{theo:arz}
    Let $(f_n)_n$, $f_n : [0, \infty) \to \Rd$ be a sequence of functions equicontinuous in an extended sense, such that the sequence $(f_n(0))_n$ is uniformly bounded. Then there exists a subsequence $(f_{n_k})_k$ and continuous function $f : [0, \infty) \to \Rd$ such that
    \begin{equation*}
        f_{n_k} \xrightarrow{k \to \infty} f \text{ uniformly on the interval $[0, T]$, for every $T\,.$}
    \end{equation*}
\end{theo}
The aim is therefore to prove that sequences $(X_n)_n$ and $(Z_n)_n$ satisfy assumptions of Theorem~\ref{theo:arz}, and then prove that if any converging subsequence converges to some $(X, Z)$, then $(X, \frac{d}{dt}Z)$ is a solution of~\eqref{eq:ode}. The gap in the proof of~\cite[Theorem 5.2.1]{Kushner} appears in the moment of applying Theorem~\ref{theo:arz}, as the equicontinuity of sequences $(X_n)_n$ and $(Z_n)_n$ is not sufficiently discussed in there. The author's investigation of the rest of the proof revealed no significant gap.

The proof that sequences $(X_n)_n$ and $(Z_n)_n$ are equicontinuous is the subject of Theorem~\ref{theo:main}. For the sake of completeness, we provide Proposition~\ref{prop:satode}, which allows identifying the limits of subsequences of $((X_n, Z_n))$ with solutions to~\eqref{eq:ode}. Upon necessary stability conditions imposed on~\eqref{eq:ode}, we can then obtain convergence of $(x_n)_n$ to the stationary points of the ODE. See Theorem~\ref{theo:final} for the precise formulation. 
\section{Results}\label{sec:res}
Our main result, formulated in Theorem~\ref{theo:main}, concerns equicontinuity of the sequences $(X_n)_n$ and $(Z_n)_n$ defined in~\eqref{eq:def-f}. As stated in Section~\ref{sec:meth}, this kind of result is an important step in examining the convergence of~\eqref{eq:def-SA}. In particular, it is exactly a missing step in the proof of convergence in the book of Kushner and Yin~\cite[Theorem V.2.1]{Kushner}. The theorem reads as follows.

%
\begin{theo}\label{theo:main}
    Let the sequences $(\gamma_n)_n$, $(e_n)_n$ and $(r_n)_n$ satisfy~\eqref{eq:con-gamma} and~\eqref{eq:con-er}.
    Let $K \subseteq \Rd$ is a hyperrectangle, $d \in \N$, $h : K \to \Rd$ be bounded, $x_0 \in \Rd$ be arbitrary.
    Let the sequence $(x_n)_n$ be defined by~\eqref{eq:def-SA}, sequence $(X_n)_n$ by~\eqref{eq:def-f}, and sequence $(Z_n)_n$ as in~\eqref{eq:def-Z}. Then, with probability $1$, sequences $(X_n)_n$ and $(Z_n)_n$ are equicontinuous in the extended sense. Moreover, if a subsequence $(Z_{n_k})_k$ converges to a function $Z$, in the sense of Theorem~\ref{theo:arz}, then $Z$ is Lipschitz.
\end{theo}
The proof of Theorem~\ref{theo:main} is relatively elementary and relies on key insights regarding projections onto a hyperrectangle. Essentially, proving equicontinuity involves managing the differences between the elements of the sequence $(x_n)_n$, which can be decomposed into two parts: the contribution from the projections and the contribution from the summation of the terms $\gamma_n(h(x_n) + e_n + r_n)$. Controlling the latter is not difficult, as it primarily depends on the assumptions~\eqref{eq:con-gamma} and~\eqref{eq:con-er} and the boundedness of the function $h$ on set $K$.

However, estimating the impact of the projections is more complicated. The fact that $K$ is a hyperrectangle plays a crucial role in this part of the proof. The structure of the hyperrectangle allows us to treat projections on each coordinate independently. A critical observation is that if two consecutive non-trivial projections are close to each other, they must project to the same point. Intuitively, this implies that if a sequence of non-trivial projections occurs within a short time frame, the projections — and consequently the iterates — cannot differ significantly.

Conversely, if projections are infrequent, the proof simplifies, as equicontinuity is more straightforward to establish without projections. The Lipschitz continuity of the limit of the sequence $(Z_{n_k})$ follows through a similar reasoning. We emphasize that the assumption that $K$ is a hyperrectangle cannot be skipped easily in our approach.

Before showing the actual proof of Theorem~\ref{theo:main}, let us prove the lemma concerning the equicontinuity of functions defined similarly as in~\eqref{eq:def-f} and~\eqref{eq:def-Z}.
\begin{lem}\label{lem:equic}
    Let the sequence $(\gamma_n)_n$ be a sequence of positive numbers satisfying~\eqref{eq:con-gamma}. For every $n$, define $t_n \coloneqq \sum_{k=1}^{n} \gamma_k$. Let $(y_n)_n$ be any sequence in $\Rd$ and define the sequence of functions $(f_n)_n$, $f_n : [0, \infty) \to \Rd$, as
    \begin{equation*}
        f_1(t) = \sum_{i=1}^{k} y_i\,, \quad \text{if } t \in [t_{k-1}, t_k), \quad f_n(t) = f_1(t + t_{n-1})\,.
    \end{equation*}
    Then, the following two assertions are equivalent.
    \begin{itemize}
        \item[(i)] Sequence $(f_n)_n$ is equicontinuous in the extended sense;
        
        \item[(ii)] For any $\ve > 0$, there exist $\delta > 0$ and $N \in \N$ such that
        \begin{equation*}
            \left| \sum_{k = n}^{m} y_k \right| < \ve\,, \text{ for every $n, m$, $N \leq n \leq m$, such that } \sum_{k=n}^{m} \gamma_k < \delta\,. 
        \end{equation*}
    \end{itemize}
\end{lem}
\begin{proof}
Let us start with proving the implication $(i) \Rightarrow (ii)$. Let us take any $\ve > 0$. By equicontinuity of $(f_n)_n$, there exist $\delta \in (0, 1)$ and $N \in \N$ such that for any $n \geq N$ and $t, s \leq 1$ such that $|t - s| < \delta$, it holds that
\begin{equation*}
    |f_n(t) - f_n(s)| < \ve\,.
\end{equation*}
Let us take any $n \geq N$ and $m \geq n$ such that $\sum_{k=n}^{m} \gamma_k < \delta$. We then have $t_{m-1} - t_{n-1} < \delta$, and therefore
\begin{equation*}
    \left| \sum_{k=n}^{m} y_k \right| = |f_n(t_{m-1} - t_{n-1}) - f_n(0)| < \ve\,. 
\end{equation*}
Thus, the implication $(i) \Rightarrow (ii)$ is proven. To prove the reverse implication, let us fix any $\ve > 0$ and $T > 0$. There exist $\delta > 0$ and $N \in \N$ such that for every $n \geq N$, it holds that
\begin{equation}\label{eq:lip30ass}
    \left| \sum_{k=n}^{m} y_k \right| < \ve\,, \text{ whenever } \quad \sum_{k=n}^{m} \gamma_k < \delta\,.
\end{equation}
Let us take any $n \geq N$ and $s, t \geq 0$ such that $s \leq t$ and $t - s \leq \tfrac{\delta}{2}$. Note that by the definition of $f_n$, there exist numbers $i, j \in \N$ such that
\begin{equation}\label{eq:lip30}
    f_n(s) = \sum_{k=1}^{i-1} y_k\,, \quad f_n(t) = \sum_{k=1}^{j} y_k\,.
\end{equation}
Indeed, we can take $i$ to be the lowest number such that $s < t_{i-1}$ and $j$ to be the greatest number such that $t_{j-1} \leq t$. Then, we get that $[t_{i-1}, t_{j-1}] \subseteq [s, t]$, and therefore
\begin{equation*}
    \sum_{k=i}^{j-1} \gamma_k \leq t - s \leq \frac{\delta}{2}\,.  
\end{equation*}
As the sequence $(\gamma_n)_n$ converges to $0$, for sufficiently large $n$ we have
\begin{equation*}
    \sum_{k=i}^{j} \gamma_k \leq \frac{\delta}{2} + \gamma_j < \delta\,.  
\end{equation*}
This in conjunction with~\eqref{eq:lip30ass} and~\eqref{eq:lip30} implies that
\begin{equation*}
    |f_n(t) - f_n(s)| = \left| \sum_{k=i}^{j} y_k \right| < \ve\,.
\end{equation*}
Thus, the implication $(ii) \Rightarrow (i)$ is proven.
\end{proof}
%

We are now in a position to prove Theorem~\ref{theo:main}.
\begin{proof}[Proof of Theorem~\ref{theo:main}]
    Let us denote $K = \Pi_{i=1}^{d} [a^i, b^i]$, where $a^1, \dots, a^d, b^1, \dots, b^d \in \R$, $a^i < b^i$ for every $i$. By formula~\eqref{eq:projtheo2} and Lemma~\ref{lem:equic}, to prove equicontinuity of the sequence $(X_n)_n$, it suffices to prove that almost surely, for every $\ve > 0$ there exists a $\delta > 0$ such that for sufficiently large $n, m$ it holds that
    \begin{equation}\label{eq:goal}
        \left| \sum_{k=n}^{m} \left(\gamma_ke_k + \gamma_kr_k + \gamma_kh(x_k) - P_k\right) \right| < \ve\,, \text{ whenever } \quad \sum_{k=n}^{m} \gamma_k < \delta\,.
    \end{equation}
    Note that as $h$ is bounded, there exists a constant $H > 0$ such that $|h(x)| \leq H$ for any $x \in K$, and therefore $|h(x_k)| \leq H$ for any $k$. Note also that almost surely the sequence $(|r_n|)$ is bounded by some constant $R > 0$. Fix any $\ve > 0$ and put $\delta_1 \coloneqq \frac{\ve}{2(H+R)}$. Whenever $\sum_{k=n}^{m} \gamma_k < \delta_1$, we have
    \begin{equation}\label{eq:lip1}
        \left|\sum_{k=n}^{m} \gamma_kh(x_k) + \gamma_kr_k \right| \leq (H+R)\sum_{k=n}^{m} \gamma_k \leq \frac{\ve}{2}\,. 
    \end{equation}
    On the other hand, by~\eqref{eq:con-er}, almost surely we have that for sufficiently large $n$ it holds that
    \begin{equation}\label{eq:lip2}
        \left|\sum_{k=n}^{m} \gamma_ke_k\right| < \frac{\ve}{2}\,.
    \end{equation}
    By~\eqref{eq:lip1} and~\eqref{eq:lip2}, we have that almost surely for any $\ve > 0$ there exist a $\delta_{\ve} > 0$ and $N_{\ve} \in \N$ such that we have
    \begin{equation}\label{eq:lip3}
        \left|\sum_{k=n}^{m} \gamma_k(e_k + r_k + h(x_k))\right| < \ve\,, \text{ whenever $N_{\ve} \leq n \leq m$ and } \sum_{k=n}^{m} \gamma_k < \delta_\ve \,.
    \end{equation}
    We shall now focus on estimating the sums $\sum_{k=n}^{m} P_k$. Let us fix a coordinate $l \in \{1, \dots, d\}$. We take the subsequence $(P_{n_j}^l)_j$ of all elements of the sequence $(P_{n}^l)_n$ that are non-zero, i.e., projection on $l$-th coordinate is not trivial. Note that for any $j$ we have
    \begin{equation}\label{eq:lip5}
        P_{n_{j+1}}^l = x_{n_j}^l - x_{n_{j+1}}^l + \sum_{k=n_j}^{n_{j+1} - 1} \gamma_k(e_k^l + r_k^l + h^l(x_k))\,, 
    \end{equation}
    as all the projections on $l$-th coordinate are zero between $n_j$ and $n_{j+1}$. Moreover, as in each point $n_j$, the projection is non-zero, we have that $x_{n_j}^l$ is equal to $a^l$ or $b^l$. Let us take $\ve = \frac{b^l - a^l}{2}$, and recall $\delta_{\ve}$, $N_{\ve}$ from~\eqref{eq:lip3}. Observe that for any $j$ such that $n_j \geq N_\ve$, it holds that
    \begin{equation}\label{eq:lip4}
        \text{if } \sum_{k=n_j}^{n_{j+1} - 1} \gamma_k < \delta_\ve\,,\quad \text{ then }\quad x_{n_j}^l = x_{n_{j+1}}^l\,.
    \end{equation}
    Indeed, reasoning by contradiction, suppose without loss of generality that $x_{n_j}^l = a^l$ and $x_{n_{j+1}}^l = b^l$. As at the point $n_{j+1}$ we made a non-zero projection on $l$-th coordinate, the point before projection has to be larger than $b^l$. This, together with the definition of $\delta_\ve$ gives us that
    \begin{equation*}
        b^l \leq x_{n_j}^l + \sum_{k=n_j}^{n_{j+1} - 1} \gamma_k(e_k^l + r_k^l + h^l(x_k)) \leq a^l + \frac{b^l - a^l}{2} < b^l\,,
    \end{equation*}
    which is a contradiction. Therefore,~\eqref{eq:lip4} is proven. Let us now fix any $\ve > 0$ and recall $\delta_{\ve / 2}$ in~\eqref{eq:lip3}. Using identity~\eqref{eq:lip5} and implication~\eqref{eq:lip4}, we obtain that for sufficiently large $n, m$ such that $\sum_{k=n}^{m} \gamma_k < \delta_{\ve/2}$, we have
    \begin{equation}\label{eq:lip6}
        \sum_{k=n}^{m} P_k^l = \sum_{k=i}^{j} P_{n_k}^l = P_{n_j}^l + \sum_{k=i}^{j-1} \sum_{s=n_k}^{n_{k+1} - 1} \gamma_s(e_s^l + r_s^l + h^l(x_s)) = P_{n_j}^l + \sum_{k=n_i}^{n_j - 1} \gamma_k(e_k^l + r_k^l + h^l(x_k))\,,
    \end{equation}
    where $i$ is the smallest number such that $n_i \geq n$ and $P_{n_i}^l \neq 0$, and $j$ is the biggest number such that $n_j \leq m$ and $P_{n_j}^l \neq 0$. By~\eqref{eq:con-gamma} and~\eqref{eq:con-er}, almost surely for sufficiently large $k$ we have $\gamma_k|e_k| \leq \frac{\ve}{4}$. Moreover, as $\gamma_k \xrightarrow{k \to \infty} 0$, we have that $|\gamma_k| \leq \frac{\ve}{4(H+R)}$, for sufficiently large $k$, and therefore
    \begin{equation*}
        |\gamma_k(e_k + r_k + h(x_k))| \leq \frac{\ve}{4} + (H+R) \cdot \frac{\ve}{4(H+R)} = \frac{\ve}{2}\,.
    \end{equation*}
    Thus, by the definition of the projection, it holds that
    \begin{align}\label{eq:lip28-4}
        |P_k^l| \leq |P_k| &= \dist\left(x_k + \gamma_k(e_k + r_k + h(x_k)), K\right) \nonumber \\ &\leq \dist(x_k + \gamma_k(e_k + r_k + h(x_k)), x_k) = |\gamma_k(e_k + r_k + h(x_k))| \leq \frac{\ve}{2}\,.
    \end{align}
    This in conjunction with~\eqref{eq:lip3} and~\eqref{eq:lip6} means that
    \begin{equation*}
        \left| \sum_{k=n}^{m} P_k^l \right| \leq |P_{n_j}^l| + \left|  \sum_{k=n_i}^{n_j - 1} \gamma_k(e_k^l + r_k^l + h^l(x_k)) \right| \leq \ve\,.
    \end{equation*}
    We therefore proved that for any $\ve > 0$, there exists $\delta > 0$ such that for sufficiently large $n, m$ it holds that
    \begin{equation*}
        \left| \sum_{k=n}^m P_k^l \right| < \ve\,, \text{ whenever } \quad \sum_{k=n}^{m} \gamma_k < \delta\,.
    \end{equation*}
    Since it is proven for any $l$ separately, using the inequality
    \begin{equation}
        \left| \sum_{k=n}^m P_k \right| \leq \sum_{l=1}^{d} \left| \sum_{k=n}^m P_k^l \right|\,,
    \end{equation}
    we obtain that for any $\ve > 0$, there exists $\delta > 0$ such that for sufficiently large $n$ it holds that
    \begin{equation}\label{eq:lip28}
        \left| \sum_{k=n}^m P_k \right| < \ve\,, \text{ whenever } \quad \sum_{k=n}^{m} \gamma_k < \delta\,.
    \end{equation}
    This, in conjunction with~\eqref{eq:lip3} and triangle inequality, gives us~\eqref{eq:goal}. By Lemma~\ref{lem:equic}, this gives us equicontinuity of the sequence $(X_n)_n$, while equicontinuity of $(Z_n)_n$ follows directly from~\eqref{eq:lip28}. \newline

    Now, we shall focus on proving the second part of the theorem. Let us take any subsequence $(Z_{n_k})$ converging to $Z$ uniformly on every interval. Let us fix any $s, t \geq 0$ and then fix any $\ve > 0$. We shall at first show that for sufficiently large $N$, it holds that
    \begin{equation}\label{eq:lip28-2}
        |Z_N(t) - Z_N(s)| \leq Hd|t-s| + \ve\,.
    \end{equation}
    By the definition of the sequence $(Z_n)_n$, there exist $n$ and $m$ such that
    \begin{equation*}
        \sum_{k=n}^{m-1} \gamma_k \leq |t - s| \quad \text{and} \quad |Z_N(t) - Z_N(s)| = \left| \sum_{k=n}^{m} P_k \right|\,.
    \end{equation*}
    Indeed, we may take $n$ to be the smallest number such that $s < t_{n-1} - t_{N-1}$, and $m$ to be the largest number such that $t_{m-1} - t_{N-1} \leq t$. Note that then $n \geq N$. Let us fix any $l \in \{1, \dots, d\}$. By the last display and~\eqref{eq:lip6}, we have
    \begin{equation}\label{eq:sie1-2}
        |Z_N^l(t) - Z_N^l(s)| \leq |P_{n_j}^l| + \left| \sum_{k=n_i}^{n_j - 1} \gamma_ke_k^l  \right| + \left| \sum_{k=n_i}^{n_j - 1} \gamma_k(h^l(x_k) + r_k) \right| \leq |P_{n_j}^l| + \left| \sum_{k=n_i}^{n_j - 1} \gamma_ke_k^l  \right| + (H+R)|t-s|\,.
    \end{equation}
    By assumption~\eqref{eq:con-gamma} and~\eqref{eq:con-er}, for sufficiently large $N$ it holds that
    \begin{equation}\label{eq:sie1}
        \left| \sum_{k=n_i}^{n_j - 1} \gamma_ke_k^l  \right| < \frac{\ve}{2d}\,.
    \end{equation}
    Following the same inequalities as in~\eqref{eq:lip28-4}, we have that $|P_{n_j}^l| < \tfrac{\ve}{2d}$. This in conjunction with~\eqref{eq:sie1} and~\eqref{eq:sie1-2} gives us that
    \begin{equation*}
        |Z_N^l(t) - Z_N^l(s)| \leq \frac{\ve}{d} + (H+R)|t-s|\,.
    \end{equation*}
    As $l$ is arbitrary, we have that
    \begin{equation*}
        |Z_N(t) - Z_N(s)| \leq \sum_{l=1}^{d} |Z_N^l(t) - Z_N^l(s)| \leq \ve + (H+R)d|t-s|\,.
    \end{equation*}
    By passing to the limit with $N \to \infty$, we obtain $|Z(t) - Z(s)| \leq \ve + (H+R)d|t-s|$, and as $\ve$ is arbitrary, we have $|Z(t) - Z(s)| \leq (H+R)d|t-s|$. Therefore, function $Z$ is Lipschitz.
\end{proof}
We shall now present the next steps of the proof of~\cite[Theorem V.2.1]{Kushner}. Equicontinuity of the sequence $(X_n)$, proven in Theorem~\ref{theo:main}, allows us to use Arzel\`a-Ascoli theorem~\ref{theo:arz}. The idea is to prove that any limit function $X$ solves the projected ODE corresponding to~\eqref{eq:def-SA}. As functions $(X_n)_n$ are built from the sequence $(x_n)_n$, there is a relation between the behavior of $(x_{n})_n$ at infinity and the behavior of limiting functions $X$. We note that by Theorem~\ref{theo:main}, the limit $Z$ possesses a weak derivative. Let us, therefore, prove that $X$ is a solution to the projected ODE.
\begin{prop}\label{prop:satode}
Let assumptions of Theorem~\ref{theo:main} be satisfied and let $h$ be continuous. Let $f$ and $Z$ be limits of some subsequences $(X_{n_k})$ and $(Z_{n_k})$, and let $z$ be the weak derivative of $Z$. Then the pair $(X, z)$ satisfies the following projected ODE
\begin{equation}\label{eq:projodee}
    \dot x = h(x) - z\,, \quad z \in N_K(x)\,.
\end{equation}
\end{prop}
\begin{proof}
We shall start by proving the identity
\begin{equation}\label{ode:int}
    X(t) = X(0) + \int_{0}^{t} h(X(s))\,ds - Z(t) + Z(0) \quad \text{ for all $t \geq 0$.}
\end{equation}
Let us take any $t \geq 0$. For any $k \in \N$, let $m_k \in \N$ be the greatest number such that $t_{m_k-1} - t_{n_k-1} \leq t$. Then, we have
\begin{equation}\label{eq:lip31}
    X_{n_k}(t) = x_{m_k} = x_{n_k} + \sum_{i=n_k}^{m_k-1} (\gamma_ih(x_i) + \gamma_ie_i + \gamma_ir_i - P_i)\,.
\end{equation}
For every $i$ such that $n_k \leq i \leq m_k - 1$, function $(X_{n_k})$ is constantly equal to $x_i$ on an interval of length $\gamma_i$. Thus
\begin{equation}\label{eq:lip31-2}
    \sum_{i=n_k}^{m_k-1} \gamma_ih(x_i) = \int_{0}^{t_{m_k-1} - t_{n_k-1}} h(X_{n_k}(s))\,ds =  \int_{0}^{t} h(X_{n_k}(s))\,ds - \int_{t_{m_k-1} - t_{n_k-1}}^{t} h(X_{n_k}(s))\,ds\,.
\end{equation}
As $X$ is uniform limit of $(X_{n_k})$ on $[0, t]$, and $h$ is continuous, we have
\begin{equation}\label{eq:lip31-3}
    \int_{0}^{t} h(X_{n_k}(s))\,ds \xrightarrow{k \to \infty} \int_{0}^{t} h(X(s))\,ds\,.
\end{equation}
Note that by the definition of $m_k$ it holds that $t - (t_{m_k-1} - t_{n_k - 1}) \leq \gamma_{m_k}$. Moreover, the sequence $(h \circ X_n)_n$ is uniformly bounded, as $h$ is bounded. This means that
\begin{equation*}
    \int_{t_{m_k-1} - t_{n_k-1}}^{t} h(X_{n_k}(s))\,ds \xrightarrow{k \to \infty} 0\,,
\end{equation*}
which in conjunction with~\eqref{eq:lip31-2} and~\eqref{eq:lip31-3} means that
\begin{equation}\label{eq:lip31-4}
    \sum_{i=n_k}^{m_k-1} \gamma_ih(x_i) \xrightarrow{k \to \infty} \int_{0}^{t} h(f(s))\,ds\,.
\end{equation}
Analogously, we have
\begin{equation*}
    \sum_{i=n_k}^{m_k-1} P_i = Z_{n_k}(t_{m_k - 1} - t_{n_k - 1}) - Z_{n_k}(0) = Z_{n_k}(t_{m_k - 1} - t_{n_k - 1}) - Z_{n_k}(t) + Z_{n_k}(t) - Z_{n_k}(0)\,. 
\end{equation*}
As $(Z_{n_k})_k$ is equicontinuous and $t - (t_{m_k-1} - t_{n_k - 1}) \xrightarrow{k \to \infty} 0$, we have 
\begin{equation}\label{eq:lip31-5}
    \sum_{i=n_k}^{m_k-1} P_i  \xrightarrow{k \to \infty} Z(t) - Z(0)\,.
\end{equation}
We can also estimate
\begin{equation}\label{eq:help4-10}
    \left|\sum_{i=n_k}^{m_k - 1} \gamma_ir_i\right| \leq \left(\sup_{i \geq n_k} |r_i|\right) \sum_{i=n_k}^{m_k - 1} \gamma_i \leq \left(\sup_{i \geq n_k} |r_i|\right)t \xrightarrow{k \to \infty} 0\,.
\end{equation}
As $\sum_{n=1}^{\infty} \gamma_ne_n$ is convergent almost surely, we have that $\sum_{i=n_k}^{m_k-1} \gamma_ie_i \xrightarrow{k \to \infty} 0$. This in conjunction with~\eqref{eq:lip31-4},~\eqref{eq:lip31-5}, and~\eqref{eq:help4-10}, gives us that
\begin{equation*}
    \sum_{i=n_k}^{m_k-1} (\gamma_ih(x_i) + \gamma_ie_i + \gamma_ir_i - P_i) \xrightarrow{k \to \infty} \int_{0}^{t} h(X(s))\,ds - Z(t) + Z(0)\,.
\end{equation*}
Therefore, by passing to the limit in~\eqref{eq:lip31}, we get
\begin{equation*}
    X(t) = X(0) + \int_{0}^{t} h(X(s))\,ds - Z(t) + Z(0)\,,
\end{equation*}
which is~\eqref{ode:int}. Let us now focus on proving that $(X, Z - Z(0))$ satisfies other requirements to be a solution to the integral form of projected ODE. From Theorem~\ref{theo:main}, function $Z$ is Lipschitz continuous. Therefore, for any $s, t \in \R$, we have
\begin{equation}\label{eq:sie1-3}
    |X(t) - X(s)| = \left| \int_{s}^{t} h(X(\tau))\,d\tau + Z(s) - Z(t)  \right| \leq H|t-s| + L_Z|t - s|\,, 
\end{equation}
where $H > 0$ is a constant such that $|h| \leq H$, and $L_Z$ is a Lipschitz constant of $Z$. By~\eqref{eq:sie1-3}, function $X$ is Lipschitz continuous. We get that almost everywhere $X$ and $Z$ have derivatives $\dot x$ and $z$, respectively, that satisfy
\begin{equation}\label{eq:help7-10}
    \dot x(t) = h(\dot x(t)) - z(t)\,.
\end{equation}
Note that $P_n \in N_K(x_n)$ for every $n$. Indeed, by the definition of projection, we have that for every $z \in K$, it holds that $\langle P_n, x_n - z \rangle \geq 0\,$, which by the definition of normal cone~\eqref{eq:defnorm} means that $P_n \in N_K(x_n)$. Therefore, for any $n, m \in \N$, we have
\begin{equation}\label{eq:sie1-4}
    \sum_{i=n}^{m} P_i \in \text{co} \left[ \bigcup_{i=n}^{m} N_K(x_i) \right]\,.
\end{equation}
Let us fix $s, t \geq 0$, $t \geq s$. For any $k \in \N$, let us denote by $j_k$ the smallest number such that $s < t_{j_k-1} - t_{n_k - 1}$, and by $m_k$, the largest number such that $t_{m_k - 1} - t_{n_k - 1} \leq t$, so that
\begin{equation*}
    Z_{n_k}(t) - Z_{n_k}(s) = \sum_{i=j_k}^{m_k} P_i\,.
\end{equation*}
Note that $X_{n_k}([s, t]) = \{x_{j_k - 1}, x_{j_k}, \dots, x_{m_k}\}$, which in conjunction with~\eqref{eq:sie1-4} gives us that
\begin{equation}\label{eq:sie1-5}
    Z_{n_k}(t) - Z_{n_k}(s) \in \text{co} \left[ \bigcup_{i=j_k}^{m_k} N_K(x_i) \right] \subseteq \text{co} \left[ \bigcup_{\tau \in [s, t]} N_K(X_{n_k}(\tau)) \right]\,.
\end{equation}
Since $X$ is a uniform limit of $(X_{n_k})$, for any $\delta > 0$, for sufficiently large $k$ we have $\sup_{\tau \in [s, t]}|X_{n_k}(\tau) - X(\tau)| < \delta$, which means
\begin{equation}\label{eq:help31oct}
    \text{co} \left[ \bigcup_{\tau \in [s, t]} N_K(X_{n_k}(\tau)) \right] \subseteq \text{co} \left[ \bigcup_{\tau \in [s, t]} \bigcup_{y \in B_{\delta}(X(\tau))} N_K(y) \right]\,.
\end{equation}
We shall now use the upper semi-continuity property of the normal cone~\cite[Proposition 6.5]{Rockafellar}. This property in conjunction with continuity of $X$ and~\eqref{eq:help31oct}, allows us to pass to the limit in~\eqref{eq:sie1-5}, i.e., to obtain
\begin{equation}\label{eq:sie1-6}
    Z(t) - Z(s) \in \bigcap_{\delta > 0} \text{co} \left[ \bigcup_{\tau \in [s, t]} \bigcup_{y \in B_{\delta}(X(\tau))} N_K(y) \right] \subseteq \text{co} \left[ \bigcup_{\tau \in [s, t]} N_K(X(\tau)) \right]\,.
\end{equation}
Note that whenever $Z$ is differentiable in $t$, then by~\eqref{eq:sie1-6}, we have
\begin{equation}\label{eq:help31-2}
    z(t) \in \bigcap_{\delta > 0} \text{co} \left[ \bigcup_{\tau \in [t-\delta, t+\delta]} N_K(X(\tau)) \right] = N_K(X(t))\,,
\end{equation}
where the last equality is again a consequence of~\cite[Proposition 6.5]{Rockafellar}. Assertion~\eqref{eq:help31-2} together with~\eqref{eq:help7-10} implies the result.
%
%
\end{proof}
\section{Applications}\label{sec:app}
The most straightforward application of Theorem~\ref{theo:main} and Proposition~\ref{prop:satode} is that they together fill the gap in the proof of~\cite[Theorem V.2.1]{Kushner}. That is, one can establish the convergence of~\eqref{eq:def-SA} under conditions of Theorem~\ref{theo:main}, provided that related projected ODE~\ref{eq:ode} has fine stability properties. Let us provide the precise theorem.
\begin{theo}\label{theo:final}
    Let the assumptions of Theorem~\ref{theo:main} be satisfied and let function $h$ be continuous. Let the function $V : K \to [0, \infty)$ be $C^1$ and such that for every $x \in \Rd$ it holds that
    \begin{equation*}
        \langle \nabla V(x), \Pi_{T_K}(h(x)) \rangle \leq 0\,.
    \end{equation*}
    Assume further that the set $S \coloneqq \{x \in K : \langle \nabla V(x), \Pi_{T_K}(h(x)) \rangle = 0\}$ has empty interior. Then almost surely $\dist(x_n, S)\xrightarrow{n \to \infty} 0$.
\end{theo}
The proof is technical and follows essentially the same lines as proof of~\cite[Theorem 3.5]{Majewski}. The natural choice of the function $V$ in Theorem~\ref{theo:final} is a Lyapunov function for the projected ODE~\eqref{eq:ode}. Then, the set $S$ is a set of stationary states of this ODE. The existence of such a function requires that all those stationary states are stable in the sense that all solutions converge to some of them at infinity. Then, Theorem~\ref{theo:final} allows us to conclude that the sequence $(x_n)_n$ converges to the set of stationary states. \newline

Let us now point out further consequences of our results. Following~\cite{Majewski}, our approach may also be used to analyze the convergence of the stochastic proximal gradient algorithm. The stochastic proximal gradient algorithm is an extension of the Proximal Gradient method, used when the function's gradient can't be computed precisely and is affected by some errors. Specifically, we want to minimize the composite function of the form
\[
P(x) = f(x) + g(x),
\]
where \( f \) is a continuously differentiable function defined on some open set \( \Xset \subseteq \mathbb{R}^d \), and \( g \) is a locally Lipschitz, bounded from below, and regular function in the sense of Clarke gradient (see \cite{Clarke}).

In many machine learning applications, \( f \) represents the empirical risk of the model, and \( g \) represents penalties that promote sparsity, like LASSO \cite{Tibshirani2011Regression}, MCP \cite{Zhang2010Nearly}, or SCAD \cite{fan2001Variable}. All penalties mentioned above are locally Lipschitz, bounded from below, and regular function in the sense of Clarke gradient. 

We look at situations where the gradient \( \nabla f(x_k) \) cannot be computed precisely, but noisy estimates \( H_k \) of \( \nabla f(x_k) \) are available. It is well known that extra conditions are needed, even without noise, to ensure that the iterates stay within a compact set and eventually reach stationary points. To solve this, we use a projected version of the algorithm.

Let \( \mathbb{I}_K \) denote the convex indicator function of the set \( K \), defined by
\[
\mathbb{I}(x) = \begin{cases}
0 & \text{if } x \in K, \\
\infty & \text{otherwise}.
\end{cases}
\]
Additionally, let \( \prox \) denote the proximal operator, which is defined as
\[
\prox_{\gamma g}(x) = \arg\min_{y \in \Xset} \left\{ g(y) + \frac{1}{2\gamma} \|x - y\|^2 \right\}.
\]
We consider two versions of the projected stochastic proximal gradient algorithm, given by the following updates:

\begin{align}\label{stoch1}
x_k &\in \prox_{\gamma_k (g + \mathbb{I}_K)} \left( x_{k-1} - \gamma_k H_k \right)\,, \\
x_k &\in \Pi_K \left( \prox_{\gamma_k g} \left( x_{k-1} - \gamma_k H_k \right) \right)\,.\label{stoch2}
\end{align}

These two projection approaches are not equivalent. Depending on the properties of \( g \) and \( K \), one may be easier to compute than the other.

By combining Theorem~\ref{theo:main} with \cite[Theorem 5.4]{Majewski}, we obtain the following convergence result:

\begin{theo}\label{thm:stoch-prox-grad-conv}Let $K \subseteq \Rd$ is a compact hyperrectangle. 
Assume that \( \Xset \subset \mathbb{R}^d \) is an open set, \( f: \Xset \to \mathbb{R} \) is a continuously differentiable function, and \( g: \Xset \to \mathbb{R} \) is a locally Lipschitz, bounded from below, regular function in the sense of the Clarke gradient. Denote by \( \delta_k \) the gradient perturbation:
\[
\delta_k = H_k - \nabla f(x_{k-1}) \,.
\]
Assume further that \( \delta_k \) can be decomposed as \( \delta_k = e_k + r_k \), where $(e_k),(r_k)$ satisfy \eqref{eq:con-er} and $(\gamma_k)$ satisfies \eqref{eq:con-gamma}. Let
\[
\mathcal{S} = \left\{ x \in K : 0 \in \nabla f(x) + \bar{\partial} g(x) - \normalcone{K}(x) \right\},
\]
where \( \bar{\partial} g \) is the Clarke gradient of \( g \) (see \cite{Clarke}), and \( \normalcone{K}(x) \) is the normal cone to the set \( K \) at \( x \in K \). Suppose that \( (f + g)(\mathcal{S}) \) has empty interior. Then:

\begin{enumerate}[(i)]
\item The sequence \( \{x_k\} \) generated by iterations \eqref{stoch1} converges to \( \mathcal{S} \).
\item If, in addition, \( g \) is Lipschitz, then the sequence \( \{x_k\} \) generated by iterations \eqref{stoch2} also converges to \( \mathcal{S} \).
\end{enumerate}
\end{theo}

\begin{proof}
The proof follows the same arguments as in \cite{Majewski}, with the only difference being that the equicontinuity of the relevant sequences and the Lipschitz continuity of the limit are provided by Theorem~\ref{theo:main}.
\end{proof}
We note that our result does not require the boundedness of \( \delta_k \), which represents a significant improvement over \cite[Theorem 5.4]{Majewski}. We observe that Theorem~\ref{theo:final} and Theorem~\ref{thm:stoch-prox-grad-conv} are special cases of the general framework of differential inclusions considered in \cite{Majewski}. Our approach could also be easily adapted to the abovementioned framework.   In particular, \cite[Theorem 4.1 and 4.2]{Majewski} hold without the assumption of bounded noise. The only part of the argument in \cite{Majewski} where boundedness is required can be easily replaced by Theorem~\ref{theo:main}.

\printbibliography

\end{document}